\documentclass[twoside]{article}
\usepackage{fancyhdr}
\title{\LARGE{\bfseries{On existence of thermally coupled incompressible
flows in a system of three dimensional pipes}}}
\date{}
\author{\large{
Michal Bene\v s\footnote{Department of Mathematics,
Faculty of Civil Engineering, Czech Technical University in Prague,
Th\'{a}kurova 7, 166 29 Prague 6, Czech Republic,
E-mail: benes@mat.fsv.cvut.cz}\;
and
Igor Pa\v{z}anin\footnote{Department of Mathematics, Faculty of Science, University of Zagreb,
Bijeni\v{c}ka 30, 10000 Zagreb, Croatia,
E-mail: pazanin@math.hr}
}}

\usepackage{amssymb}
\usepackage{amsfonts}
\usepackage{graphicx}
\usepackage{amsmath}
\usepackage{latexsym}
\usepackage{amssymb}
\usepackage{pstricks}
\usepackage{amsmath}
\usepackage{graphics}
\usepackage{mathrsfs}
\usepackage{fancyhdr}
\usepackage{times}

 \newtheorem{thm}{Theorem}[section]
 \newtheorem{cor}[thm]{Corollary}
 \newtheorem{lem}[thm]{Lemma}

 \newtheorem{rem}[thm]{Remark}
\numberwithin{equation}{section}

\newtheorem{problem}[thm]{Problem}

\newenvironment{proof}[1][Proof.]{\begin{trivlist}
\item[\hskip \labelsep {\bfseries #1}]}{\end{trivlist}}

\newcommand{\om}{\Omega}
\newcommand{\pom}{\partial\Omega}

\newcommand{\dc}{\|}

\newcommand{\iom}{\int_\om}

\newcommand{\bfx}{\mbox{\boldmath{$x$}}}
\newcommand{\bfn}{\mbox{\boldmath{$n$}}}

\newcommand{\bff}{\mbox{\boldmath{$f$}}}

\newcommand{\bfu}{\mbox{\boldmath{$u$}}}
\newcommand{\bfw}{\mbox{\boldmath{$w$}}}

\newcommand{\bfv}{\mbox{\boldmath{$v$}}}

\newcommand{\bfphi}{\mbox{\boldmath{$\phi$}}}

\newcommand{\bfpsi}{\mbox{\boldmath{$\psi$}}}

\newcounter{constants}
\begin{document}
\maketitle

%
%
%
%
%
%
%
%
%

%



\begin{abstract}
We study an initial-boundary-value problem for time-dependent flows
of heat-conducting viscous incompressible fluids
in a system of three-dimensional pipes on a time interval $(0,T)$.
Here we are motivated by the bounded domain approach
with ``do-nothing'' boundary conditions.
In terms of the velocity, pressure and enthalpy of the fluid,
such flows are described by a parabolic system with
strong nonlinearities and including the artificial
boundary conditions for the velocity and nonlinear boundary
conditions for the so called enthalpy of the fluid.
The present analysis is devoted to the proof of the existence of weak solutions
for the above problem.
In addition, we deal with some regularity for the velocity of the fluid.
\end{abstract}


\section{Introduction}

Many problems of fluid thermo-mechanics involving unbounded domains occur in many areas
of applications, e.g. flows of a liquid in duct systems, fluid flows through a thin or long pipe
or through a system of pipes in hemodynamics and so on. From a numerical point of view,
these formulations are not convenient and quite practical.
Therefore, an efficient natural way  is to cut off unbounded parts of the domain
by introducing an artificial boundary in order to limit the computational work.
Then the original problem posed in an unbounded
domain is approximated by a problem in a smaller bounded computational region
with artificial boundary conditions prescribed at the cut boundaries.
Hence, let $\om$ be a bounded domain in $\mathbb{R}^3$ with
boundary $\partial \Omega$.
In a physical sense, $\Omega$ represents a
``truncated'' region of an unbounded system of pipes occupied by a
moving heat-conducting viscous incompressible fluid. ${\Gamma_1}$
will denote the ``lateral'' surface and ${\Gamma_2}$ represents the
open parts (cut boundaries) of the piping system.
It is physically reasonable to assume that in/outflow pipe segments extend as straight pipes.
More precisely,
$\Gamma_1$ and ${\Gamma_2}$ are ${C}^{\infty}$-smooth open
disjoint not necessarily connected subsets of $\pom$ such that
$\Gamma_2=\bigcup_{i\in \mathcal{J}} \Gamma^{(i)}_2$,
$\Gamma^{(i)}_2\cap\Gamma^{(j)}_2=\emptyset$ for $i\neq j$,
$\pom =
\overline{\Gamma}_1\cup\overline{\Gamma}_2$,
${\Gamma_1}\neq\emptyset$, ${\Gamma_2}\neq\emptyset$, $\mathcal{M} =
\partial\om-({\Gamma_1}\cup{\Gamma_2}) = \overline{\Gamma}_1
\cap\overline{\Gamma}_2 =
\bigcup_{i\in \mathcal{J}} \mathcal{M}_i$, $\mathcal{J} =
\left\{1,\dots,d\right\}$, and the $2$--dimensional measure of
$\mathcal{M}$ is zero and $\mathcal{M}_i$ are smooth nonintersecting
curves {(this
means that $\mathcal{M}_i$ are smooth curved nonintersecting
edges and vertices
(conical points) on
$\partial \Omega$ are excluded)}. Moreover, all portions
of ${\Gamma_2}$ are taken to be flat
and ${\Gamma_1}$ and ${\Gamma_2}$ form a right angle
$\omega_{\mathcal{M}}=\pi/2$ at all points of $\mathcal{M}$ (in the
sense of tangential planes), see Figure \ref{pipe}.
\begin{figure}[h]
\centering
\includegraphics[width=11.0cm]{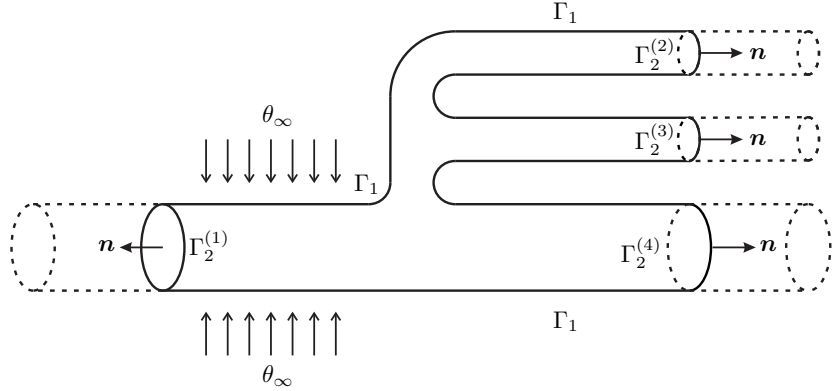}
\caption{Truncated piping system.}
\label{pipe}
\end{figure}
The flow of a viscous incompressible heat-conducting
fluid is governed by balance equations for linear
momentum, mass and internal energy \cite{diaz2007}:
\begin{align}
\varrho_0\left(\bfu_t + (\bfu\cdot\nabla)\bfu \right) - \nu\Delta \bfu
+ \nabla P
&= \rho(\theta)\bff,
\label{momentum}
\\
\nabla\cdot (\varrho_0\bfu)
&= 0,
\label{mass}
\\
c_v\rho(\theta)\left( \theta_t + \bfu\cdot\nabla\theta \right) - \lambda
\Delta \theta  -\nu \mathbb{D}(\bfu):\mathbb{D}(\bfu)
&= h.
\label{energy}
\end{align}
Here $\bfu=(u_1,u_2,u_3)$, $P$ and $\theta$ denote the unknown
velocity, pressure and temperature, respectively. Tensor
$\mathbb{D}(\bfu)$ denotes the symmetric part of the velocity gradient.
Data of the problem are as follows: $\bff$ is a body force and $h$ a
heat source term. Positive constant material coefficients represent
the kinematic viscosity $\nu$, reference density $\rho_0$, heat conductivity
$\lambda$ and specific heat at constant volume $c_v$.
Following the well-known Boussinesq approximation,
the temperature dependent density
is used in the energy equation \eqref{energy}
and to compute the buoyancy force $\rho(\theta)\bff$
on the right-hand side of equation \eqref{momentum}. Everywhere else in the model,
$\rho$ is replaced by the reference value $\varrho_0$.
Change of
density $\rho$ with temperature is given by
strictly positive, nonincreasing and continuous function, such that
\begin{equation}\label{con:rho}
0 < \rho_1 \leq   \rho(\xi) \leq \rho_2 < +\infty
\quad \forall \xi \in \mathbb{R}
\quad (\rho_1,\rho_2 = {\rm const}).
\end{equation}
The energy balance
equation \eqref{energy} takes into account the phenomena of the
viscous energy dissipation and adiabatic heat effects. For rigorous
derivation of the model like \eqref{momentum}--\eqref{energy} we
refer the readers to \cite{Kagei2000}. Rigorously derived asymptotic models describing stationary motion of heat-conducting incompressible viscous
fluid through pipe-like domains can be found in \cite{Pazanin1,Pazanin2}.

To complete the model, suitable boundary and initial conditions have to be added.
Concerning the boundary conditions of the flow, it is a standard
situation to prescribe a homogeneous no-slip boundary condition
for the velocity of the fluid on the fixed walls of the channel, i.e.
\begin{equation}\label{dirichlet_u}
\bfu= {\bf0}
\qquad
\textmd{ on } \Gamma_{1}.
\end{equation}
Since nothing is known in advance about the flow through the open parts,
it is really not clear what type of boundary
condition for the velocity should be prescribed on ${\Gamma_2}$.
The condition frequently used in numerical practice for viscous parallel flows
is the most simple outflow boundary condition of the form
\begin{equation}\label{do_nothing_nonhom}
-P \bfn_i+\nu\frac{\partial \bfu}{\partial \bfn_i}= F_i \bfn_i
\qquad
\textmd{ on }{\Gamma^{(i)}_2},
\qquad
i\in \mathcal{J},
\end{equation}
which seems to be natural
since it does not prescribe anything on the cut cross-section
of an in/outlets of the truncated region.
Therefore, this condition is usually called the ``do nothing''
(or ``free outflow'') boundary condition.
In \eqref{do_nothing_nonhom},
$\bfn_i = \bfn({\Gamma^{(i)}_2})$ is the outer unit normal vector to ${\Gamma^{(i)}_2}$,
$i\in \mathcal{J}$, while quantities $F_i$ are given functions.
In particular, for time-dependent flows, $F_i$ are given functions of time.
Boundary condition \eqref{do_nothing_nonhom} results from a
variational principle and does not have a real physical meaning.
For further discussion on theoretical aspects as well as practical difficulties
of this boundary condition and
the physical meaning of the quantities $F_i$
we refer to \cite{Benes2009,GaRaRoTu,Gresho,Hey,SaniGresho}.
\begin{rem}
Assume that $F_i$ are given smooth functions of time on ${\Gamma^{(i)}_2}$,
$i\in \mathcal{J}$,
and consider the smooth extension $F$ on ${\Omega \times {I}}$ such that
$F(\bfx,t)\big|_{{\Gamma^{(i)}_2}}\equiv F_i(t)$. Introducing the
new variable $\mathcal{P}=P+F$ this amounts to solving the problem
with the homogeneous ``do nothing'' boundary condition transferring the data
from the right-hand side of the boundary condition to the right-hand side
of the linear momentum balance equation.
Hence, for simplicity, we assume
throughout this paper, without loss of generality, that $F_i\equiv
0$, i.e.
\begin{equation}\label{homogeneous_neumann}
-P \bfn_i+\nu\frac{\partial \bfu}{\partial \bfn_i}= {\bf0}
\qquad
\textmd{ on }{\Gamma^{(i)}_2},
\qquad
i\in \mathcal{J}.
\end{equation}
\end{rem}
Concerning the heat transfer through the walls of pipes we consider
the Newton boundary condition
\begin{equation}\label{temp:newton}
-\lambda \frac{\partial \theta}{\partial \bfn}
=
\alpha(\theta - \theta_{\infty})
-
q_{\theta}
\qquad
\textmd{ on }\Gamma_{1},
\end{equation}
in which $\alpha$ designates the heat transfer coefficient,
${\theta}_\infty$ is the prescribed temperature outside the computational domain and
$q_{\theta}$ represents the heat flux imposed on the lateral surfaces.
On the open parts of the piping system we use the classical outflow
(``do nothing'') condition
\begin{equation}\label{temp:neumann}
\frac{\partial \theta}{\partial \bfn}
= 0
\qquad
\textmd{ on }\Gamma_{2}.
\end{equation}
Initial conditions are considered as the given initial velocity field $\bfu_0$
and the temperature profile $\theta_{0}$ over the flow domain
\begin{equation}\label{initial_u_theta}
\bfu(\bfx,0) = \bfu_0(\bfx)
\;  \textmd{ and } \;
\theta(\bfx,0) = \theta_0(\bfx)
\qquad
\textmd{ in } \Omega.
\end{equation}
\begin{rem}
Obtained results in this paper can be extended to problems with
Dirichlet or the mixed (Dirichlet-Neumann) boundary conditions for
the temperature on the walls.
Namely, instead of \eqref{temp:newton}, we can consider
($\Gamma_1 = \overline{\Gamma_3} \cup \overline{\Gamma_4}$,
${\Gamma_3} \cap {\Gamma_4} = \emptyset$)
\begin{align*}
\theta &= \theta_D
\qquad
\textmd{ on }\Gamma_{3},
\\
-\lambda \frac{\partial \theta}{\partial \bfn}
&=
\alpha(\theta - \theta_{\infty})
-
q_{\theta}
\qquad
\textmd{ on }\Gamma_{4}.
\end{align*}
\end{rem}
The paper is organized as follows. In Section
\ref{sec:preliminaries}, we introduce basic notations and some
appropriate function spaces in order to precisely formulate our
problem. Furthermore,
we rewrite the energy equation by using the appropriate enthalpy transformation.
In Section \ref{main_result}, we present the strong
form of the model for the non-stationary motion of viscous
incompressible heat-conducting fluids in a system of 3D pipes considered in our
work, specify our smoothness assumptions on data and formulate the problem in a
variational setting.
We also provide the bibliographic remarks on the subject
and indicate what kind of difficulties we should overcome in the process.
The main result, the existence of strong-weak solutions,
stated at the end of Section \ref{main_result},
is proved in Section~\ref{sec:proof}.
The proof rests on application of Schauder fixed point theorem.
First, we present basic results on the
existence and uniqueness of solutions to auxiliary problems, the decoupled initial-boundary value
problems for the non-stationary Stokes system with mixed boundary conditions
and the parabolic convection-diffusion equation with the nonlinear boundary condition.
In the proof of the main result we rely on the energy estimates for auxiliary
problems, regularity of stationary solutions to the Stokes problem
and interpolations-like inequalities.


%
\section{Preliminaries. Enthalpy transformation}
\label{sec:preliminaries}
Vectors and vector functions are denoted by boldface letters.
Throughout the paper, we will always use positive constants $C$,
$c$, $c_1$, $c_2$, $\dots$, which are not specified and which may
differ from line to line but do not depend on
the functions under consideration.
Throughout this paper we suppose
$p,q,p'\in [1,\infty]$, $p'$ denotes the conjugate exponent
to $p>1$, ${1}/{p} + {1}/{p'} = 1$.
$L^p(\Omega)$ denotes the usual
Lebesgue space equipped with the norm $\|\cdot\|_{L^p(\Omega)}$ and
$W^{k,p}(\Omega)$, $k\geq 0$ ($k$ need not to be an integer, see
\cite{KufFucJoh1977}), denotes the usual
Sobolev-Slobodecki space with the norm $\|\cdot\|_{W^{k,p}(\Omega)}$.
In case of vector-valued functions we shall use the notation
$\boldsymbol L^p:= L^p(\Omega)^3$ and similarly for
other function spaces.
Recall that
$W^{0,p}(\Omega)=L^p(\Omega)$.
In the paper we shall use the following embedding theorems
(see \cite{AdamsFournier1992,KufFucJoh1977}):
\begin{equation}
\label{embedding_theorems}
\left\{
\begin{array}{lll}
W^{k,p}(\Omega) \hookrightarrow L^q(\Omega),
&
\|\varphi\|_{L^q(\Omega)} \leq c \,\|\varphi\|_{W^{k,p}(\Omega)},
&
1\leq q <\infty, \, kp=3,
\\
W^{k,p}(\Omega) \hookrightarrow L^q(\Omega),
&
\|\varphi\|_{L^q(\Omega)} \leq c \,\|\varphi\|_{W^{k,p}(\Omega)},
&
1\leq q \leq 3p/(3-kp),
\\
& & kp<3,
\\
W^{k,p}(\Omega) \hookrightarrow L^{\infty}(\Omega),
&
\|\varphi\|_{L^{\infty}(\Omega)} \leq c \,\|\varphi\|_{W^{k,p}(\Omega)},
&
kp>3
\end{array}\right.
\end{equation}
for every $\varphi \in W^{k,p}(\Omega)$.
{
By \cite{cwikel},
\begin{equation}\label{compact_embedding}
W^{\mu,2}(\Omega) \hookrightarrow\hookrightarrow W^{\mu-\varepsilon,2}(\Omega),
\qquad \textmd{ if }\mu \geq \varepsilon >0
\end{equation}
(the symbol ``$\hookrightarrow\hookrightarrow$'' denotes the compact embedding).}
Further, there exists
the continuous operator
$\mathcal{R}: W^{k,p}(\Omega) \rightarrow L^q(\partial\Omega)$,
such that
\begin{equation*}
\|\mathcal{R}(\varphi)\|_{L^q(\partial\Omega)}
\leq c
\,\|\varphi\|_{W^{k,p}(\Omega)}
\; \forall \varphi \in W^{k,p}(\Omega) \;
\left\{
\begin{array}{lll}
1 \leq p < 3/k,\,
q = \frac{3p-p}{3-kp}, 
\\
p \geq \max \left\{ 1,3/k \right\},\,
q  \in [1,\infty).
\end{array}\right.
\end{equation*}
In what follows we often omit $\Omega$ in notations
of spaces and norms if it causes no ambiguity.
Unless specified otherwise, we use
Einstein's summation convention for indices running from $1$ to $3$.
Further, let
\begin{equation*}
{E}_u
:=
\left\{
\bfu\big|_{\Omega};\;\bfu\in C_c^\infty(\mathbb{R}^3)^3; \,
\nabla\cdot\bfu = 0,
\,
{\textmd{supp}\, \bfu}  \cap (\partial\Omega\backslash \Gamma_{2})
= \emptyset
\right\}
\end{equation*}
and ${V}_{{\Gamma_2}}^{k,p}$ be the closure of ${E}_u$ in the
norm of $(W^{k,p})^3$, $k\ge 0$ and $1\leq p \leq \infty$.
Then
${V}_{{\Gamma_2}}^{k,p}$
is a
Banach space with the norm of the space $(W^{k,p})^3$.
Further, define the space
\begin{equation}
{D} := \left\{\bfu \; |\; \bff \in {V}_{{\Gamma_2}}^{0,2},\,
a_u(\bfu,\bfv) =(\bff,\bfv)  \textmd{ for all } \bfv\in
{V}_{{\Gamma_2}}^{1,2} \right\} \label{D_u}
\end{equation}
equipped with the norm
\begin{equation*}
\|\bfu\|_{{D}} := \|\bff\|_{{V}_{{\Gamma_2}}^{0,2}},
\end{equation*}
where $\bfu$ and $\bff$ are corresponding functions via \eqref{D_u}.
\begin{lem}\label{lemma:emb_D}
The following embedding holds for the space  $D$:
\begin{equation*}
D \hookrightarrow \mathbf{W}^{2,2},
\qquad
\|\bfu\|_{\mathbf{W}^{2,2}} \leqslant c(\Omega)\|\bfu\|_{D}
\quad \forall \bfu \in D.
\end{equation*}
\end{lem}
\begin{proof}
For the proof see \cite[Lemma 2.1 and Appendix A]{Benes2014a}.
$\square$
\end{proof}

In what follows we apply a so-called enthalpy transformation
$e=\mathcal{E}(\theta)$, defined by
\begin{displaymath}
\mathcal{E}(\xi) = \int\limits_{0}^{\xi} c_v\rho(s)
{\rm d}s.
\end{displaymath}
In view of \eqref{con:rho}, $\mathcal{E}(\cdot)$ is strictly increasing,
so the transformation is bijective with $\beta:=\mathcal{E}^{-1}$
Lipschitz continuous,
$\beta(0)=0$.
There exists a positive constant $C_{\beta}$ such that
\begin{equation}\label{cont_beta}
|\beta(\zeta_1) - \beta(\zeta_2)| \leq C_{\beta} |\zeta_1-\zeta_2|
\quad \forall \zeta_1,\zeta_2 \in \mathbb{R}
\end{equation}
and
\begin{equation}\label{monotony_beta}
0 <  (\beta(\zeta_1) - \beta(\zeta_2))(\zeta_1-\zeta_2)
\quad \forall \zeta_1,\zeta_2 \in \mathbb{R}, \; \zeta_1 \neq \zeta_2.
\end{equation}

Now we set $e(\bfx):=\mathcal{E}(\theta(\bfx))$ for $\bfx\in\Omega$
and introduce the following notation
\begin{align}
\beta(e) &= \theta,
\label{beta}
\\
\varrho({e}) &= \rho(\beta(e)),
\label{varrho}
\\
e_0 &= \mathcal{E}(\theta_0).
\label{}
\end{align}
In terms of $e$, the energy equation \eqref{energy} takes the form
\begin{equation*}
{e}_t + \bfu\cdot\nabla{e}
- \nabla  \left(\kappa({e})\nabla{e}\right)
- \nu\mathbb{D}(\bfu):\mathbb{D}(\bfu)
=
h,
\end{equation*}
where
\begin{equation}\label{kappa}
\kappa({e}) = \frac{\lambda}{c_v\rho(\mathcal{E}^{-1}(e))}.
\end{equation}
Here we have, in view of \eqref{con:rho},
\begin{equation}\label{con:kappa}
0 < \kappa_1 \leq   \kappa(\xi) \leq \kappa_2 < +\infty \quad
\forall \xi \in \mathbb{R}
\quad
\left(\kappa_1= \frac{\lambda}{c_v\rho_2},\;\kappa_2
= \frac{\lambda}{c_v\rho_1}\right).
\end{equation}
To simplify mathematical formulations, we introduce the following
notations:
\begin{align}
a_u(\bfu,\bfv) & :=   \iom \frac{\partial u_i}{\partial x_j}
\frac{\partial v_i}{\partial x_j}\,{{\rm d}\Omega},
\label{form_a}
\\
b_u(\bfu,\bfv,\bfw)& := \iom u_j{\frac{\partial v_i}{\partial x_j}}
w_i\,{{\rm d}\Omega},\label{form_b}
 \\
a_{e}(\eta, \phi,\varphi) &
:=
\iom  \eta
\frac{\partial \phi}{\partial x_i}  \frac{\partial \varphi}{\partial x_i}
\,{{\rm d}\Omega},
\label{form_c}
 \\
b_e(\bfu,\phi,\varphi)& :=
\iom
u_i \frac{\partial \phi}{\partial x_i}
 \,\varphi \,{{\rm d}\Omega},
 \label{form_d}
 \\
 \gamma(\phi,\varphi) & := \int_{\Gamma_1}  \mathcal{R}(\phi)  \mathcal{R}(\varphi)   \,{\rm d}{S},
  \label{form_gamma}
 \\
d(\bfu,\bfv,\varphi) & :=  \iom  d_{ij}(\bfu) d_{ij}(\bfv) \varphi
\,{{\rm d}\Omega}, \label{form_e}
 \\
(\bfu,\bfv) &:=  \iom u_i  v_i \,{{\rm d}\Omega},
 \label{scalar_Lu}
\\
(\phi,\varphi)_{\Omega} & :=  \iom \phi \varphi \,{{\rm d}\Omega},
\label{scalar_Lt}
\\
\langle g , \varphi \rangle &  :=
\int_{\Gamma_1} ( \alpha \theta_{\infty} + q_e ) \varphi   \,{\rm d}{S}
+
\iom h \varphi \,{{\rm d}\Omega}.
\label{functional_g}
\end{align}
In \eqref{form_a}--\eqref{functional_g} all functions
$\bfu,\bfv,\bfw,\eta,\phi,\varphi, \theta_{\infty},q_e, h$
are smooth enough, such that all
integrals on the right-hand sides make sense. In \eqref{form_e},
$d_{ij}(\bfu)$ denotes the components of the tensor $\mathbb{D}(\bfu)$
defined by
\begin{displaymath}
d_{ij}(\bfu)=\frac{1}{2}\left(\frac{\partial u_i}{\partial
x_j}+\frac{\partial u_j}{\partial x_i}\right), \qquad i,j=1,2,3.
\end{displaymath}
Let $T\in(0,\infty)$ be fixed throughout the paper, $I:=(0,T)$.
Let $E$ be a Banach space,
by $L^p(I;E)$ we denote the Bochner space (see \cite{AdamsFournier1992}).
The following compactness result for spaces involving time
was established by J.P.~Aubin (see
\cite{Aubin1963}) and will be crucial to prove the main proposition of the paper.
\begin{thm}\label{aubin_compact}
Let $\mathcal{B}_0$, $\mathcal{B}$, $\mathcal{B}_1$ be three Banach
spaces where $\mathcal{B}_0$, $\mathcal{B}_1$ are reflexive. Suppose
that $\mathcal{B}_0$ is continuously imbedded into $\mathcal{B}$,
which is also continuously imbedded into $\mathcal{B}_1$, and
imbedding from $\mathcal{B}_0$ into $\mathcal{B}$ is compact. For
any given $p_0$, $p_1$ with $1<p_0,p_1<\infty$, let
\begin{displaymath}
\mathcal{W}:=\left\{v ; \; v\in L^{p_0}(0,T;\mathcal{B}_0), \;
v_t\in L^{p_1}(0,T;\mathcal{B}_1)  \right\}.
\end{displaymath}
Then the imbedding from $\mathcal{W}$ into
$L^{p_0}(0,T;\mathcal{B})$ is compact.
\end{thm}
\begin{rem}
To simplify the notation, in the sequel we normalize material
constants $\alpha$, $\rho_0$, $\nu$, $\lambda$ and $c_v$ to one.
\end{rem}
\section{Variational formulation and the main result}
\label{main_result}
In view of \eqref{beta}--\eqref{kappa},
system \eqref{momentum}--\eqref{initial_u_theta} for
the new variables $(\bfu, P, e)$
transforms to
\begin{align}
\bfu_t + (\bfu\cdot\nabla)\bfu  - \Delta \bfu
+ \nabla{P}
& = \varrho({e})\bff
&&
\textmd{in}\; {\Omega \times {I}},
\label{eq3}
\\
\nabla\cdot\bfu &= 0
&&
\textmd{in}\;{\Omega \times {I}},
\label{eq5}
\\
{e}_t + \bfu\cdot\nabla{e}
- \nabla  \left(\kappa({e})\nabla{e}\right)
 - \mathbb{D}(\bfu):\mathbb{D}(\bfu)
& = h
&&
\textmd{in}\;{\Omega \times {I}},
\label{heat equation}
\\
\bfu
& = {\bf 0}
&& \textmd{on} \;\Gamma_{1}  \times {I},
\label{eq6}
\\
-\kappa({e})\frac{\partial{e}}{\partial\bfn} &= \beta(e) + q_e
&&
\textmd{on}\; \Gamma_{1}  \times {I},
\label{boundary temperature}
\\
-{P}\bfn+\frac{\partial\bfu}{\partial\bfn}&= {\bf0}
&&\textmd{on} \; \Gamma_{2}  \times {I},
\label{eq7}
\\
\frac{\partial{e}}{\partial\bfn} &= 0
&&
\textmd{on}\; \Gamma_2  \times {I},
\label{boundary temperature2}
\\
\bfu(\bfx,0) &= \bfu_0(\bfx)
&&
\textmd{in} \; \Omega,
\label{eq8}
\\
e(\bfx,0) &={e}_0(\bfx)
&&
\textmd{in} \; \Omega .
\label{init_temp}
\end{align}
We suppose that all functions in \eqref{eq3}--\eqref{init_temp} are
smooth enough.

\bigskip

At this point we can formulate our problem in  a variational sense:
\begin{problem}
 Suppose that
\begin{align*}
& \bff \in L^{2}(I;{V}_{{\Gamma_2}}^{0,2}),
\quad
g \in L^2({I}; {W}^{-1,2}),
\nonumber
\\
& \bfu_0 \in {V}_{{\Gamma_2}}^{1,2}, \quad e_0 \in L^{2}.
\nonumber
\end{align*}
Find a pair $[\bfu,e]$ such that
\begin{align*}
& \bfu_t \in L^2( {I};{V}_{{\Gamma_2}}^{0,2}),
\quad
\bfu \in L^2({I};{D}) \cap \mathcal{C}({I}; {V}_{{\Gamma_2}}^{1,2}),
\\
& e_t \in L^2( {I};{W}^{-1,2}),
\quad
e \in L^2({I};{W}^{1,2}) \cap \mathcal{C}({I};L^2)
\end{align*}
and the following system
\begin{eqnarray*}
(\bfu_t,\bfv)
+
a_u(\bfu,\bfv)
+
b_u(\bfu,\bfu,\bfv)
&=&
( \varrho(e)\bff,\bfv ),
\\
\langle{e}_t,\varphi\rangle
+
a_{e}(\kappa(e),e,\varphi)
+
b_e(\bfu,{e},\varphi)
\\
+
\gamma(\beta(e),\varphi)
-
d(\bfu,\bfu,\varphi)
&=&
\langle g,\varphi\rangle
\end{eqnarray*}
holds for every $[\bfv,\varphi] \in {V}_{{\Gamma_2}}^{1,2} \otimes
{W}^{1,2}$ and for almost every $t\in{I}$ and
\begin{align*}
\qquad \qquad \qquad \bfu(\bfx,0) & =  \bfu_0(\bfx)
&& \textmd{ in } \Omega,
\\
e(\bfx,0) & = {e}_0(\bfx)
&& \textmd{ in } \Omega.
\end{align*}
The pair $[\bfu,{e}]$ is called the strong-weak solution to the system
\eqref{eq3}--\eqref{init_temp}.
\end{problem}

\begin{rem}
The main advantage of the formulation of the Navier-Stokes
equations in free divergence spaces is that the pressure
$P$ is eliminated from the
system.
Having $\bfu$ in hand,
this unknown can be recovered  in the same way as in \cite{Kucera2009}.
\end{rem}

Let us briefly describe some difficulties we have to solve in our
work. The equations  \eqref{eq3}--\eqref{heat equation} represent
the system with strong nonlinearities (quadratic growth of
$\nabla\bfu$ in dissipative term $\mathbb{D}(\bfu):\mathbb{D}(\bfu)$)
without
appropriate general existence and regularity theory. In
\cite{frehse}, Frehse presented a simple example of discontinuous
bounded weak solution $\bfw\in L^{\infty}\cap H^1$ of the nonlinear
elliptic system of the type $\Delta\bfw=B(\bfw,\nabla\bfw)$, where
$B$ is analytic and has quadratic growth in $\nabla\bfw$. However,
for scalar problems, such existence and regularity theory is well
developed (cf. \cite{LadUr,LadSolUr}).

Nevertheless, the main (open) problem of the system
\eqref{eq3}--\eqref{init_temp} consists in the fact that, because of
the boundary condition \eqref{eq7}, we cannot prove that
$b(\bfu,\bfu,\bfu)=0$. Consequently, we are not able to show that
the kinetic energy of the fluid is controlled by the data of the
problem and solutions of \eqref{eq3}--\eqref{init_temp} need not
satisfy the energy inequality. This is due to the fact that some
uncontrolled ``backward flow'' can take place at the open parts
${\Gamma_2}$ of the domain $\Omega$ and one is not able to prove
global existence results. In
\cite{Kracmar2002}--\cite{KraNeu5}, Kra\v cmar and Neustupa
prescribed an additional condition on the output (which bounds the
kinetic energy of the backward flow) and formulated steady and
evolutionary Navier-Stokes problems by means of appropriate
variational inequalities. In \cite{KuceraSkalak1998}, Ku\v cera and
Skal\' ak proved the local-in-time existence and uniqueness of a
variational solution of the Navier-Stokes equations for
iso-thermal fluids, such that
\begin{equation*}\label{reg_kucera_skalak}
\bfu_t\in L^2({I}^*; {V}_{{\Gamma_2}}^{1,2}),
\quad \bfu_{tt}\in L^2({I}^*;{V}_{{\Gamma_2}}^{-1,2}), \quad 0<T^*\leq T,
\end{equation*}
under some smoothness restrictions on $\bfu_0$ and ${P}$. In
\cite{SkalakKucera2000}, the same authors established similar
results for the Boussinesq approximations of the
heat--conducting incompressible fluids. In \cite{Kucera2009},
Ku\v{c}era supposed that the ``do nothing'' problem for the
Navier-Stokes system is solvable in suitable function class with
some given data. The
author proved that there exists a unique solution for data which are
small perturbations of the original ones.

In case of isothermal flows, in \cite{Benes2011c},
the first author of the present paper proved
local-in-time existence and uniqueness
of regular solutions to isothermal Navier-Stokes flows for Newtonian fluids
in three-dimensional non-smooth domains
with various types of boundary conditions,
such that
$\bfu\in L^2(0,T^*;{D})$,  ${D}\hookrightarrow W^{2,2}(\Omega)^3$,
which is regular in the sense that solutions possess second spatial derivatives.
In \cite{Benes2014a}, the same author proved the local-in-time existence,
global uniqueness and smoothness of the solution of an initial-boundary-value problem
for Boussinesq flows in three-dimensional channel-like domains
excluding viscous dissipation and considering constant density in the energy balance equation.
In case of corresponding stationary flows, the existence, uniqueness and regularity of the solution has been recently proved in \cite{Benes2014b}.
In this paper, we extend the existence result from \cite{Benes2014b} to non-stationary problems.
The following theorem represents the main result of this paper.
\begin{thm}\label{theorem:existence_result}
Assume
\begin{align*}
& \bff \in L^{2}({I};{V}_{{\Gamma_2}}^{0,2}),
\quad
g \in L^2({I}; {W}^{-1,2}),
\\
& \bfu_0 \in {V}_{{\Gamma_2}}^{1,2}, \quad e_0 \in L^{2}(\Omega),
\end{align*}
and
\begin{equation}\label{main_cond}
\frac{1}{4 C_S^2 T^{1/8}}
\geq
\varrho_2 \|\bff\|_{L^{2}({I};{V}_{{\Gamma_2}}^{0,2})}
+
\|\bfu_0\|_{{V}_{{\Gamma_2}}^{1,2}}.
\end{equation}
Then there exists the strong-weak solution $[\bfu,{e}]$ to problem
\eqref{eq3}--\eqref{init_temp}. In \eqref{main_cond}, $C_S$ is
some specific constant depending on $\Omega$ (cf. \eqref{est:max_reg_stokes_2}).
\end{thm}

\section{Proof of the main result}
\label{sec:proof}

To prove Theorem \ref{theorem:existence_result}
we apply the following
\begin{thm}[Schauder fixed point theorem]\label{th:Schauder}
Let ${S}$ be a closed convex set in
a Banach space ${B}$ and let $\mathcal{T}$ be a continuous
mapping of ${S}$ into itself such that the image
$\mathcal{T}(S)$ is precompact.
Then $\mathcal{T}$ has a fixed point.
\end{thm}
For the proof of Theorem \ref{th:Schauder} see
\cite[p. 279, Theorem 11.1 and p. 280, Corollary 11.2]{GiTru}.

Before we proceed to  prove the main result of this paper, let us
establish the following well-possedness results for the auxiliary problems,
in particular, the existence of the unique regular
solution to the non-stationary Stokes problem with the mixed boundary conditions
and the existence and uniqueness of
the weak solution to the parabolic convection-diffusion equation
with the nonlinear boundary condition.

\begin{thm}[Stokes problem]\label{stokes_problem}
Let $\bff\in L^2({I};{V}_{{\Gamma_2}}^{0,2})$
and $\bfu_0 \in {V}_{{\Gamma_2}}^{1,2}$.
Then there exists the
unique function $\bfu\in L^2({I};{D}) \cap
{\mathcal{C}({I};{V}_{{\Gamma_2}}^{1,2})}$, $\bfu_t\in
L^2({I};{V}_{{\Gamma_2}}^{0,2})$, such that
\begin{equation} \label{lin_var_form_1a}
(\bfu_t,\bfv) + a_u(\bfu,\bfv) = (\bff,\bfv)
\end{equation}
holds for every $\bfv\in {V}_{{\Gamma_2}}^{1,2}$ and for almost every
$t\in{I}$ and
$$
\bfu(\bfx,0) = \bfu_0(\bfx) \qquad\textmd{ in }\Omega.
$$
Moreover,
\begin{multline}\label{eq11a}
\|\bfu_t\|_{L^2({I};{V}_{{\Gamma_2}}^{0,2})}
+
\|\bfu\|_{L^2({I};{D})}
+
\|\bfu\|_{\mathcal{C}({I};{V}_{{\Gamma_2}}^{1,2})}
\\
\leq
c(\Omega)
\left(
\|\bff\|_{L^2({I};{V}_{{\Gamma_2}}^{0,2})}
+
\|\bfu_0\|_{{V}_{{\Gamma_2}}^{1,2}}
\right).
\end{multline}
\end{thm}
\begin{proof}
The proof can be handled in exactly the same way
as the proof of \cite[Theorem 3.4]{BeKuc2014}
concerning, in particular, a two-dimensional problem.
The proof is based on the
Galerkin approximation with spectral basis and the uniform
boundedness of approximate solutions in suitable spaces.
In fact, the proof is independent
of dimension, therefore we skip it here.
$\square$
\end{proof}

By the standard parabolic-equation theory \cite[Chapter 8]{Roubicek2005} we have the following
\begin{thm}[Convection-diffusion problem with the nonlinear boundary condition]\label{linear_heat_problem}
Let $h\in L^2({I};{W}^{-1,2})$, $\eta \in L^{\infty}(I,L^{\infty}(\Omega))$,
$e_0 \in L^{2}$,
$\bfv \in L^2({I};{D}) \cap {\mathcal{C}({I};{V}_{{\Gamma_2}}^{1,2})}$.
Then there exists the unique
function ${e} \in  L^2({I};{W}^{1,2}) \cap {\mathcal{C}({I};L^2)}$,
${e}_t\in L^2({I};{W}^{-1,2})$, such that
\begin{equation*}
\langle {e}_t,\varphi \rangle
+
a_{{e}}(\kappa(\eta),{e},\varphi)
+
b_e(\bfv,{e},\varphi)
+
\gamma(\beta(e),\varphi)
=
\langle h,\varphi \rangle
\end{equation*}
holds for every $\varphi\in {W}^{1,2}$ and for almost every
$t\in{I}$ and
$$
e(\bfx,0) = e_0(\bfx) \qquad\textmd{ in }\Omega.
$$
\end{thm}


Now we are prepared to prove the main result of this paper.

\begin{proof}[Proof of Theorem \ref{theorem:existence_result}]
First, let us introduce the following reflexive Banach spaces
\begin{displaymath}
X:=
\bigl\{\bfphi\;|\;
\bfphi \in L^4({I};\mathbf{L}^{24})
\cap
L^{8}({I};\mathbf{W}^{1,24/11})
\bigr\}
\end{displaymath}
and
\begin{displaymath}
{Y} :=
\left\{\bfpsi\;|\;
\bfpsi\in L^2({I};{D})  \cap \mathcal{C}({I};{V}_{{\Gamma_2}}^{1,2}),\;
\bfpsi_t\in L^2({I};{V}_{{\Gamma_2}}^{0,2}) \right\},
\end{displaymath}
respectively, equipped with the norms
\begin{displaymath}
\|\bfphi\|_{X} :=
\dc\bfphi\dc_{L^4({I};\mathbf{L}^{24})} +
\dc\bfphi\dc_{L^{8}({I};\mathbf{W}^{1,24/11})}
\end{displaymath}
and
\begin{displaymath}
\|\bfpsi\|_{{Y}}
:=
\|\bfpsi\|_{L^2({I};{D})}
+
\|\bfpsi\|_{\mathcal{C}({I};{V}_{{\Gamma_2}}^{1,2})}
+
\dc\bfpsi_t\dc_{L^2({I};{V}_{{\Gamma_2}}^{0,2})}.
\end{displaymath}
Let us present some properties of $X$ and $Y$.
Let $\bfphi \in {Y}$. Raising and integrating the interpolation
inequality
\begin{equation*}
\| \bfphi(t) \|_{\mathbf{W}^{3/2,2}}
\leq
c(\Omega)
\| \bfphi(t) \|^{1/2}_{\mathbf{W}^{1,2}} \|\bfphi(t)\|^{1/2}_{\mathbf{W}^{2,2}}
\end{equation*}
over $I$ we get, recall ${D}\hookrightarrow \mathbf{W}^{2,2}$,
\begin{eqnarray*}
 \| \bfphi \|_{L^{4}({I};\mathbf{W}^{3/2,2})}
&\leq&
c(\Omega)
\| \bfphi \|^{1/2}_{{L}^{2}({I};\mathbf{W}^{2,2})}
\| \bfphi \|^{1/2}_{\mathcal{C}({I};\mathbf{W}^{1,2})}
\nonumber
\\
&\leq& c(\Omega)  \| \bfphi \|_{{Y}}.
\end{eqnarray*}
Hence, we have
\begin{equation}\label{emb_160}
{Y} \hookrightarrow L^{4}({I};\mathbf{W}^{3/2,2}).
\end{equation}
In view of the imbedding relations \eqref{embedding_theorems}
and \eqref{compact_embedding} for Sobolev spaces we have
\begin{equation*}
\mathbf{W}^{3/2,2}
\hookrightarrow\hookrightarrow
\mathbf{W}^{11/8,2}
\hookrightarrow
\mathbf{L}^{24}
\end{equation*}
and applying Theorem \ref{aubin_compact} we get
\begin{equation}\label{comp_emb_14}
{Y}  \hookrightarrow \hookrightarrow L^4({I};\mathbf{W}^{11/8,2})
\hookrightarrow L^4({I};\mathbf{L}^{24}).
\end{equation}
Further, raising and integrating the interpolation inequality (cf.
\cite[Theorem 5.2]{AdamsFournier1992})
\begin{equation*}
\| \bfphi(t) \|_{\mathbf{W}^{5/4,2}}
\leq
c(\Omega)
\|\bfphi(t)\|^{1/4}_{\mathbf{W}^{2,2}}
\| \bfphi(t) \|^{3/4}_{\mathbf{W}^{1,2}}
\end{equation*}
from $0$ to $T$ we get
\begin{eqnarray*}
\| \bfphi \|_{L^{8}({I};\mathbf{W}^{5/4,2})}
&\leq&
c(\Omega)
\| \bfphi \|^{1/4}_{{L}^{2}({I};\mathbf{W}^{2,2})}
\| \bfphi \|^{3/4}_{\mathcal{C}({I};\mathbf{W}^{1,2})}
\\
&\leq&
c(\Omega) \, \| \bfphi \|_{Y}.
\end{eqnarray*}
Hence,
\begin{equation}\label{emb_150}
{Y} \hookrightarrow L^{8}({I};\mathbf{W}^{5/4,2}).
\end{equation}
Note that
\begin{equation}\label{emb_151}
\mathbf{W}^{5/4,2}
\hookrightarrow\hookrightarrow
\mathbf{W}^{9/8,2}
\hookrightarrow
\mathbf{W}^{1,24/11}.
\end{equation}
Now \eqref{emb_150} and \eqref{emb_151} and Theorem \ref{aubin_compact}
yield the compact embedding
\begin{equation}\label{comp_emb_15}
{Y}
\hookrightarrow \hookrightarrow
L^8({I}; \mathbf{W}^{9/8,2} )
 \hookrightarrow
L^8({I};\mathbf{W}^{1,24/11})
.
\end{equation}
Finally, \eqref{comp_emb_14} and \eqref{comp_emb_15} imply the compact
embedding
\begin{equation}\label{comp_emb_YX}
{Y}\hookrightarrow\hookrightarrow X.
\end{equation}
Moreover,
\begin{equation}\label{est_emb_YX}
\|\bfphi\|_{X}
\leq
c(\Omega)
\|\bfphi\|_{Y}
\end{equation}
for every $\bfphi \in Y$.
Now, \eqref{est_emb_YX} and Theorem \ref{stokes_problem}
immediately  yield the  following
\begin{cor}\label{corolary_stokes}
There exists a constant $C_S$ depending only on $\Omega$,
such that the solution of \eqref{lin_var_form_1a} satisfies the estimate
\begin{equation}\label{est:max_reg_stokes_2}
\|\bfu\|_{X}
\leq
C_S
\left(
\|\bff\|_{L^2({I};{V}_{{\Gamma_2}}^{0,2})}
+
\|\bfu_0\|_{{V}_{{\Gamma_2}}^{1,2}}
\right).
\end{equation}
\end{cor}

For an arbitrary fixed couple
$$
[\widetilde{\bfu},\tilde{e}]
\in
X
\otimes
{L^{2}(I;L^{2})}
$$
we now consider the nonlinear problem
\begin{eqnarray}
({\bfu}_t,\bfv) + a_u({\bfu},\bfv)  &=& ( \varrho(\tilde{e})\bff,\bfv )
-b_u(\widetilde{\bfu},\widetilde{\bfu},\bfv),
\label{linear_problem_101}
\\
\langle{e}_t,\varphi\rangle
+
a_{e}(\kappa(\tilde{e}),e,\varphi)
+
b_e(\bfu,{e},\varphi)
\nonumber
\\
+
\gamma(\beta(e),\varphi)
-
d(\bfu,\bfu,\varphi)
&=&
\langle g,\varphi\rangle
\label{linear_problem_102}
\end{eqnarray}
for every $[\bfv,\varphi]\in {V}_{{\Gamma_2}}^{1,2} \otimes
{W}^{1,2}$ and for almost every $t\in{I}$ and
\begin{eqnarray}
{\bfu}(\bfx,0) &=& \bfu_{0}(\bfx)
\;\textmd{ in }\Omega,
\label{linear_problem_103}
\\
e(\bfx,0)&=&  e_0(\bfx) \;\textmd{ in }\Omega.
\label{linear_problem_104}
\end{eqnarray}
Let us estimate the right-hand side in \eqref{linear_problem_101}.
Applying Sobolev embeddings \eqref{embedding_theorems}
together with the Young's inequality we deduce
\begin{eqnarray}\label{eq40}
\|b_u(\tilde{\bfu},\tilde{\bfu},\cdot)\|_{L^2( {I};{V}_{{\Gamma_2}}^{0,2})}
&\leq&
\left(\int_0^T \|\tilde{\bfu}\|^2_{\mathbf{L}^{24}(\om)}
\|\tilde{\bfu}\|^2_{\mathbf{W}^{1,24/11}(\om)}{\rm d}t\right)^{1/2}
\nonumber
\\
&\leq& T^{1/8}
\|\tilde{\bfu}\|_{L^4(I;\mathbf{L}^{24}(\om))}
\|\tilde{\bfu}\|_{L^8(I;\mathbf{W}^{1,24/11}(\om))}
\nonumber
\\
&\leq&  T^{1/8}\|\tilde{\bfu}\|^2_{X}.
\end{eqnarray}
Now,
by virtue of \eqref{con:rho} and \eqref{varrho},
we can write
\begin{eqnarray}\label{eq41}
\|( \varrho(\tilde{e})\bff,\cdot )
-
b_u(\widetilde{\bfu},\widetilde{\bfu},\cdot)\|_{L^2({I};{V}_{{\Gamma_2}}^{0,2})}
&\leq&
\|( \varrho(\tilde{e})\bff,\cdot )\|_{L^2({I};{V}_{{\Gamma_2}}^{0,2})}
\nonumber
\\
&&
+
\|b_u(\widetilde{\bfu},\widetilde{\bfu},\cdot)\|_{L^2({I};{V}_{{\Gamma_2}}^{0,2})}
\nonumber
\\
&\leq&
\varrho_2\|\bff\|_{L^2({I};{V}_{{\Gamma_2}}^{0,2})}
\nonumber
\\
&&
+
T^{1/8}
\|\widetilde{\bfu}\|_{L^4(I;\mathbf{L}^{24})}
\|\widetilde{\bfu}\|_{L^8(I;\mathbf{W}^{1,24/11})}
\nonumber
\\
&\leq&
\varrho_2\|\bff\|_{L^2({I};{V}_{{\Gamma_2}}^{0,2})}
+
T^{1/8}\|\widetilde{\bfu}\|^2_{X}.
\end{eqnarray}
Hence, for $[\widetilde{\bfu},\tilde{e}]
\in X \otimes {L^{2}(I;L^{2})}$
we have
$( \varrho(\tilde{e})\bff,\cdot )
-
b_u(\widetilde{\bfu},\widetilde{\bfu},\cdot)
\in
L^2({I};{V}_{{\Gamma_2}}^{0,2})
$
and
Theorem \ref{stokes_problem} yields the existence of the unique solution $\bfu$
to problem \eqref{linear_problem_101} and \eqref{linear_problem_103},
such that
\begin{align*}
& \bfu_t \in L^2( {I};{V}_{{\Gamma_2}}^{0,2}),
\quad
\bfu \in L^2({I};{D}) \cap \mathcal{C}({I}; {V}_{{\Gamma_2}}^{1,2}).
\end{align*}
Now with $\bfu$ in hand, we define $e$ to be the solution
of problem \eqref{linear_problem_102} and \eqref{linear_problem_104},
i.e. $e$ satisfying
\begin{equation}\label{linear_problem_202}
\langle{e}_t,\varphi \rangle
+
a_{e}(\kappa(\tilde{e}),{e},\varphi)
+
b_e(\bfu,{e},\varphi)
+
\gamma(\beta({e}),\varphi)
=
\langle g,\varphi\rangle
+
d(\bfu,\bfu,\varphi)
\end{equation}
for every $\varphi \in {W}^{1,2}$ and for almost every $t\in{I}$ and
\begin{equation}\label{linear_problem_204}
e(\bfx,0)=  e_0(\bfx) \;\textmd{ in }\Omega.
\end{equation}
Using
\eqref{embedding_theorems}
it is easy to see that the dissipative term on the right
hand side of \eqref{linear_problem_202} satisfies the estimate
\begin{eqnarray}\label{est:301}
\int_0^T
|
d(\bfu,\bfu,\varphi)
|
\,
{\rm d}t
&\leq&
\int_0^T
\|\mathbb{D}(\bfu)\|^2_{\mathbf{L}^{12/5}}
\|\varphi\|_{L^6}
{\rm d}t
\nonumber
\\
&\leq&
c\,
\|{\bfu}\|^2_{L^4(I;\mathbf{W}^{1,12/5})}
\|\varphi\|_{L^2(I;{W}^{1,2})}
\nonumber
\\
&\leq&
c\,
\|{\bfu}\|^2_{Y}
\|\varphi\|_{L^2(I;{W}^{1,2})},
\end{eqnarray}
where we have used \eqref{emb_160} and the embeddings
\begin{equation}\label{emb:sup1}
\mathbf{W}^{3/2,2} \hookrightarrow \mathbf{W}^{1,3} \hookrightarrow \mathbf{W}^{1,12/5}.
\end{equation}
Hence,
for given $\bfu \in Y$ and $\tilde{e}\in {L^{2}(I;L^{2})}$
we have
$$
\langle g,\varphi\rangle
+
d(\bfu,\bfu,\varphi)
\in
L^2({I};{W}^{-1,2})
$$
and the existence of the unique weak solution
${e} \in  L^2({I};{W}^{1,2}) \cap {\mathcal{C}({I};L^2)}$,
$e_t \in L^2({I};{W}^{-1,2}),$
to \eqref{linear_problem_202}--\eqref{linear_problem_204}
follows from Theorem \ref{linear_heat_problem}.

\medskip

Let $\mathcal{T}$ denote the mapping defined by the equation
$$
[\bfu,e] = \mathcal{T}([\tilde{\bfu},\tilde{e}]).
$$
We have shown that the mapping
$$
\mathcal{T}:X
\otimes
{L^{2}(I;L^{2})}
\rightarrow
X
\otimes
{L^{2}(I;L^{2})}
$$
is well defined. We now prove that $\mathcal{T}$ is continuous.
Let
$$
[\tilde{\bfu},\tilde{e}] \in X \otimes {L^{2}(I;L^{2})}
$$
and $[\tilde{\bfu}_n,\tilde{e}_n]$ be a sequence in $X \otimes {L^{2}(I;L^{2})}$
such that
$$
[\tilde{\bfu}_n,\tilde{e}_n] \rightarrow [\tilde{\bfu},\tilde{e}]
 \textmd{ in }
X \otimes {L^{2}(I;L^{2})}.
$$
Let  $[\bfu,e] = \mathcal{T}([\tilde{\bfu},\tilde{e}])$ and
$[\bfu_n,e_n] = \mathcal{T}([\tilde{\bfu}_n,\tilde{e}_n])$.
Writing \eqref{linear_problem_101} for $[\tilde{\bfu},\tilde{e}]$
and $[\tilde{\bfu}_n,\tilde{e}_n]$ separately and
subtracting their respective equations we get
(in view of \eqref{eq11a})
\begin{multline}\label{est:ineq100}
\|\bfu - \bfu_n\|_{Y}
\leq
c
\|( \varrho(\tilde{e})\bff,\cdot ) - ( \varrho(\tilde{e}_n)\bff,\cdot )
\\
-
b_u(\tilde{\bfu},\tilde{\bfu},\cdot)
+
b_u(\tilde{\bfu}_n,\tilde{\bfu}_n,\cdot)\|_{L^2({I};{V}_{{\Gamma_2}}^{0,2})}.
\end{multline}
Now, let us simply modify the right-hand side to obtain
\begin{multline}\label{est:ineq101}
\|( \varrho(\tilde{e})\bff,\cdot ) - ( \varrho(\tilde{e}_n)\bff,\cdot )
-
(b_u(\tilde{\bfu},\tilde{\bfu},\cdot)
-b_u(\tilde{\bfu}_n,\tilde{\bfu}_n,\cdot))\|_{L^2({I};{V}_{{\Gamma_2}}^{0,2})}
\\
\leq
\| (\varrho(\tilde{e})-\varrho(\tilde{e}_n))\bff,\cdot )\|_{L^2({I};{V}_{{\Gamma_2}}^{0,2})}
+
\| b_u(\tilde{\bfu},\tilde{\bfu}-\tilde{\bfu}_n,\cdot) \|_{L^2({I};{V}_{{\Gamma_2}}^{0,2})}
\\
+
\|  b_u(\tilde{\bfu}-\tilde{\bfu}_n,\tilde{\bfu}_n,\cdot)\|_{L^2({I};{V}_{{\Gamma_2}}^{0,2})}.
\end{multline}
For convective terms on the right-hand side in \eqref{est:ineq101} we can write
\begin{eqnarray}\label{est:ineq102}
\| b_u(\tilde{\bfu},\tilde{\bfu}-\tilde{\bfu}_n,\cdot) \|_{L^2({I};{V}_{{\Gamma_2}}^{0,2})}
&\leq&
\left(
\int_0^T \|\tilde{\bfu}\|^2_{\mathbf{L}^{24}}
\|\tilde{\bfu}-\tilde{\bfu}_n\|^2_{\mathbf{W}^{1,24/11}}{\rm d}t
\right)^{1/2}
\nonumber\\
&\leq&
\|\tilde{\bfu}\|_{L^4(I;\mathbf{L}^{24})}
\|\tilde{\bfu}-\tilde{\bfu}_n\|_{L^4(I;\mathbf{W}^{1,24/11})}
 \nonumber\\
&\leq&
c \, \|\tilde{\bfu}\|_{X}\|\tilde{\bfu}-\tilde{\bfu}_n\|_{X}
\end{eqnarray}
and
\begin{eqnarray}\label{est:ineq103}
\| b_u(\tilde{\bfu}-\tilde{\bfu}_n,\tilde{\bfu}_n,\cdot)\|_{L^2({I};{V}_{{\Gamma_2}}^{0,2})}
&\leq&
\left(
\int_0^T \|\tilde{\bfu}-\tilde{\bfu}_n\|^2_{\mathbf{L}^{24}}
\|\tilde{\bfu}_n\|^2_{\mathbf{W}^{1,24/11}}{\rm d}t
\right)^{1/2}
\nonumber
\\
&\leq&
c \,
\|\tilde{\bfu}-\tilde{\bfu}_n\|_{X}\|\tilde{\bfu}_n\|_{X}.
\end{eqnarray}
Combining \eqref{est:ineq100} with the latter estimates \eqref{est:ineq101}--\eqref{est:ineq103}
we deduce
\begin{multline}\label{est:ineq104a}
\|\bfu - \bfu_n\|_{Y}
\leq
c_1 \| (\varrho(\tilde{e})-\varrho(\tilde{e}_n))\bff,\cdot )\|_{L^2({I};{V}_{{\Gamma_2}}^{0,2})}
\\
+
c_2 ( \|\tilde{\bfu}\|_X + \|\tilde{\bfu}_n\|_X)
 \|\tilde{\bfu}-\tilde{\bfu}_n\|_X.
\end{multline}
In addition, in view of \eqref{est_emb_YX}, we have
\begin{multline}\label{est:ineq104b}
\|\bfu - \bfu_n\|_{X}
\leq
c_1 \| (\varrho(\tilde{e})-\varrho(\tilde{e}_n))\bff,\cdot )\|_{L^2({I};{V}_{{\Gamma_2}}^{0,2})}
\\
+
c_2 ( \|\tilde{\bfu}\|_X + \|\tilde{\bfu}_n\|_X)
 \|\tilde{\bfu}-\tilde{\bfu}_n\|_X.
\end{multline}
In \eqref{est:ineq104a} and \eqref{est:ineq104b},
$\| (\varrho(\tilde{e})-\varrho(\tilde{e}_n))\bff,\cdot )\|_{L^2({I};{V}_{{\Gamma_2}}^{0,2})}$
 converges to zero by the Lebesgue dominated convergence theorem.
Now, estimates \eqref{est:ineq104a} and \eqref{est:ineq104b} yield
\begin{equation}\label{conv:u}
\|\bfu - \bfu_n\|_{Y} \rightarrow 0
\;\textmd{  and  }\;
\|\bfu - \bfu_n\|_{X} \rightarrow 0
\end{equation}
provided
$$
[\tilde{\bfu}_n,\tilde{e}_n] \rightarrow [\tilde{\bfu},\tilde{e}]
\textmd{  in }
X \otimes {L^{2}(I;L^{2})}.
$$
We now turn, for a moment, to the energy
balance equation \eqref{linear_problem_102}.
Using the same procedure as before, writing \eqref{linear_problem_102} for
$[\tilde{\bfu},\tilde{e}]$ and $[\tilde{\bfu}_n,\tilde{e}_n]$, respectively,
and subtracting both resulting equations yield
\begin{multline*}
\langle  (e-e_n)_t,\varphi \rangle
+
a_{e}(\kappa(\tilde{e}),{e}-e_n,\varphi)
+
\gamma(\beta({e})-\beta({e}_n),\varphi)
\\
=
a_e( \kappa(\tilde{e}_n)-\kappa(\tilde{e}) , {e}_n , \varphi )
+
b_e(\bfu_n - \bfu , {e} , \varphi)
-
b_e(\bfu_n , e_n - e , \varphi)
\\
+
d(\bfu+\bfu_n,\bfu-\bfu_n,\varphi).
\end{multline*}
The next step is to use $\varphi = e-e_n$ in order to obtain
\begin{multline}\label{est:ineq106}
\frac{d}{dt} \|e-e_n\|^2_{L^{2}}
+
\|\nabla(e-e_n)\|^2_{\mathbf{L}^{2}}
+
\gamma(\beta({e})-\beta({e}_n),e-e_n)
\\
\leq
c\left(
|(  a_e(\kappa(\tilde{e}_n)-\kappa(\tilde{e}) , {e}_n , e-e_n )|
+
|b_e(\bfu_n - \bfu , {e} , e-e_n)|
\right.
\\
+
|b_e(\bfu_n , e_n - e , e-e_n)|
\left.
+
|d(\bfu+\bfu_n,\bfu-\bfu_n,e-e_n) |
\right).
\end{multline}
Let us estimate all terms on the right-hand side in \eqref{est:ineq106}
to get
\begin{multline}\label{eq:601}
|a_e( \kappa(\tilde{e}_n)-\kappa(\tilde{e}) , {e}_n , e-e_n )|
\leq
\delta \| e-e_n \|^2_{W^{1,2}}
\\
+
C(\delta)
\|(\kappa(\tilde{e}_n) - \kappa(\tilde{e}) )\nabla{e}_n\|^2_{\mathbf{L}^2}
\end{multline}
and
\begin{eqnarray}\label{}
|b_e(\bfu_n - \bfu , {e} , e-e_n)|
&\leq&
\|\bfu_n - \bfu\|_{\mathbf{L}^4}
\|\nabla e\|_{\mathbf{L}^2}
\|e-e_n\|_{L^4}
\nonumber
\\
&\leq&
\delta \| e-e_n \|^2_{W^{1,2}}
+
C(\delta)\|\bfu_n - \bfu\|^2_{\mathbf{L}^4} \|e\|^2_{W^{1,2}},
\nonumber
\\
\end{eqnarray}
further
\begin{eqnarray}\label{}
|b_e(\bfu_n,e-e_n,e-e_n)|
&\leq&
\|\bfu_n\|_{\mathbf{L}^4}
\|\nabla (e-e_n)\|_{\mathbf{L}^2}
\|e-e_n\|_{L^4}
\nonumber
\\
&\leq&
\delta \| e-e_n \|^2_{W^{1,2}}
+
C(\delta)\|\bfu_n\|^4_{\mathbf{L}^4} \|e-e_n\|^2_{L^2}
\nonumber
\\
\end{eqnarray}
and finally
\begin{eqnarray}\label{eq:604}
|d(\bfu+\bfu_n,\bfu-\bfu_n,e-e_n) |
&\leq&
\delta \| e-e_n \|^2_{W^{1,2}}
\nonumber
\\
&&
+
C(\delta)\|
\mathbb{D}(\bfu+\bfu_n) : \mathbb{D}(\bfu-\bfu_n)
\|^2_{L^2}
\nonumber
\\
&\leq&
\delta \| e-e_n \|^2_{W^{1,2}}
\nonumber
\\
&&
+
C(\delta)
\|\bfu+\bfu_n\|^2_{\mathbf{W}^{1,4}}
\|\bfu-\bfu_n\|^2_{\mathbf{W}^{1,4}}
.
\nonumber
\\
\end{eqnarray}
In view of \eqref{monotony_beta},
choosing $\delta$ sufficiently small
and combining \eqref{est:ineq106} together with the estimates \eqref{eq:601}--\eqref{eq:604}
we deduce
\begin{eqnarray}\label{est:201}
\frac{d}{dt} \|e-e_n\|^2_{L^{2}}
+
\|e-e_n\|^2_{{W}^{1,2}}
&\leq&
\left( c_1 \|\bfu_n\|^4_{\mathbf{L}^4} + c_2 \right)\|e-e_n\|^2_{L^2}
\nonumber
\\
&&
+
c_3
\|(\kappa(\tilde{e}_n) - \kappa(\tilde{e}) )\nabla{e}_n\|^2_{\mathbf{L}^2}
\nonumber
\\
&&
+ c_4
\|\bfu_n - \bfu\|^2_{\mathbf{L}^4} \|e\|^2_{W^{1,2}}
\nonumber
\\
&&
+ c_5
\|\bfu+\bfu_n\|^2_{\mathbf{W}^{1,4}}\|\bfu-\bfu_n\|^2_{\mathbf{W}^{1,4}}.
\nonumber
\\
\end{eqnarray}
Applying the Gronwall's inequality to the estimate
\eqref{est:201} and the fact that $e(0)-e_n(0) =0$, we arrive at
\begin{equation}\label{gronwall}
\|e(t)-e_n(t)\|^2_{L^{2}}
\leq
\int_0^t \omega(s) {\rm d}s
\;
\exp\left( \int_0^t \chi(s) {\rm d}s \right)
\end{equation}
for all $0\leq t \leq T$, where
\begin{eqnarray*}
\chi(t)
&=&
c_1\|\bfu_n\|^4_{\mathbf{L}^4} + c_2,
\nonumber
\\
\omega(t)
&=&
c_3
\|(\kappa(\tilde{e}_n) - \kappa(\tilde{e}) )\nabla{e}_n\|^2_{\mathbf{L}^2}
+
c_4
\|\bfu_n - \bfu\|^2_{\mathbf{L}^4} \|e\|^2_{W^{1,2}}
\nonumber
\\
&&
+
c_5
\|\bfu+\bfu_n\|^2_{\mathbf{W}^{1,4}}\|\bfu-\bfu_n\|^2_{\mathbf{W}^{1,4}}.
\end{eqnarray*}
Here,
by the Lebesgue dominated convergence theorem and \eqref{conv:u},
$\omega(t) \rightarrow 0$ for all $0\leq t \leq T$
as $n \rightarrow \infty$ and
by \eqref{gronwall} we deduce
\begin{equation}\label{conv:e}
\|e-e_n\|_{L^2(I,L^2)} \rightarrow 0.
\end{equation}
Finally, in view of
\eqref{conv:u} and \eqref{conv:e}
we conclude
\begin{equation*}
[{\bfu}_n,{e}_n] = \mathcal{T}([\tilde{\bfu}_n,\tilde{e}_n])
\rightarrow
\mathcal{T}([\tilde{\bfu},\tilde{e}])  =  [{\bfu},{e}].
\end{equation*}
Hence, $\mathcal{T}$ is continuous.
Using Theorem \ref{aubin_compact} and the embeddings
$$
W^{1,2}\hookrightarrow\hookrightarrow L^2 \hookrightarrow {W}^{-1,2}
$$
we have
\begin{equation}\label{emb:compact_ee}
\left\{\phi;\;
\phi \in  L^2({I};W^{1,2}),
\;
\phi_t  \in L^2({I}; {W}^{-1,2}),
\right\}
\hookrightarrow\hookrightarrow
{L^{2}(I;L^{2})}.
\end{equation}
By continuity of $\mathcal{T}$
and compact embeddings \eqref{comp_emb_YX} and \eqref{emb:compact_ee},
$\mathcal{T}$ is completely continuous.

We conclude the proof by deriving some estimates of $\bfu$ and $e$.
Applying \eqref{eq11a} to linear problem \eqref{linear_problem_101} with \eqref{linear_problem_103}
and taking into account \eqref{eq40} and \eqref{eq41}
we can write for $\bfu$ the estimate
\begin{multline}\label{est:u_Y}
\|\bfu_t\|_{L^2({I};{V}_{{\Gamma_2}}^{0,2})}
+
\|\bfu\|_{L^2({I};{D})}
+
\|\bfu\|_{\mathcal{C}({I};{V}_{{\Gamma_2}}^{1,2})}
\\
\leq
c
\left(
\varrho_2\|\bff\|_{L^2({I};{V}_{{\Gamma_2}}^{0,2})}
+
T^{1/8}\|\widetilde{\bfu}\|^2_{X}
+
\|\bfu_0\|_{{V}_{{\Gamma_2}}^{1,2}}
\right).
\end{multline}
Further, following Corollary \ref{corolary_stokes} we have
\begin{equation}\label{est:aux_u_X}
\|\bfu\|_{X}
\leq
C_S
\left(
\varrho_2\|\bff\|_{L^2({I};{V}_{{\Gamma_2}}^{0,2})}
+
T^{1/8}\|\widetilde{\bfu}\|^2_{X}
+
\|\bfu_0\|_{{V}_{{\Gamma_2}}^{1,2}}
\right).
\end{equation}
Let
$$
M
:=
\left\{
\bfv \in X, \;
\|\bfv\|_X
\leq
\frac{1}{2 C_S T^{1/8}}
\right\}
$$
and $\tilde{\bfu} \in M$.
Combining \eqref{main_cond} and \eqref{est:aux_u_X} we get
\begin{eqnarray*}
\|\bfu\|_{X}
&\leq&
C_S
\left(
\varrho_2\|\bff\|_{L^2({I};{V}_{{\Gamma_2}}^{0,2})}
+
T^{1/8}\|\widetilde{\bfu}\|^2_{X}
+
\|\bfu_0\|_{{V}_{{\Gamma_2}}^{1,2}}
\right)
\nonumber
\\
&\leq&
C_S
\left(
\frac{1}{4 C_S^2 T^{1/8}}
+
T^{1/8}\frac{1}{4 C_S^2 T^{1/4}}
\right) = \frac{1}{2 C_S T^{1/8}}.
\end{eqnarray*}
Hence, $\bfu \in M$.
Let us turn, for a moment, to the equation \eqref{linear_problem_202}
with the initial condition \eqref{linear_problem_204}
and derive some estimates of $e$.
One is allowed to use  $\varphi = e$ as a test function in
\eqref{linear_problem_202} to obtain
\begin{multline}\label{est:400}
\frac{1}{2}\frac{d}{dt}\|e(t)\|^2_{L^2}
+
a_{e}(\kappa(\tilde{e}(t)),{e}(t),e(t))
+
\gamma(\beta({e}(t)),e(t))
\\
\leq
|\langle g(t),e(t)\rangle|
+
|d(\bfu(t),\bfu(t),e(t))|
+
|b_e(\bfu(t),e(t),e(t))|
\end{multline}
almost everywhere in $I$.
We are going to estimate all terms on the right-hand
side of the latter inequality.
Evidently, we have
\begin{equation}\label{est:401}
|\langle g(t),e(t)\rangle| \leq \| g(t)\|_{W^{-1,2}} \|e(t)\|_{W^{1,2}}.
\end{equation}
The dissipative term $d$ can be estimated as in \eqref{est:301} to obtain
\begin{eqnarray}\label{est:403}
|d(\bfu(t),\bfu(t),e(t))|
&\leq&
\|\mathbb{D}(\bfu(t))\|^2_{\mathbf{L}^{12/5}}\|e(t)\|_{L^{6}}
\nonumber
\\
 &\leq&
 c\|{\bfu}(t)\|^2_{\mathbf{W}^{1,12/5}}\|e(t)\|_{{W}^{1,2}}.
\end{eqnarray}
The last term in \eqref{est:400}
can be handled using the interpolation inequality and the Young's inequality to get
\begin{eqnarray}\label{est:404}
|b_e(\bfu(t),e(t),e(t))|
&\leq&
\|\bfu(t)\|_{\mathbf{L}^4}
\|\nabla e(t)\|_{\mathbf{L}^2}
\|e(t)\|_{L^4}
\nonumber
\\
&\leq&
c
\|\bfu(t)\|_{\mathbf{L}^4}
\|e(t)\|^{7/4}_{W^{1,2}}
\|e(t)\|^{1/4}_{L^2}
\nonumber
\\
&\leq&
\delta\|e(t)\|^2_{W^{1,2}}
+C(\delta)
\|\bfu(t)\|^8_{\mathbf{L}^4}
\|e(t)\|^2_{L^2}.
\end{eqnarray}
By virtue of \eqref{monotony_beta} and the fact that $\beta(0)=0$
we have $\gamma(\beta({e}(t)),e(t)) \geq 0$.
Hence, choosing $\delta$ sufficiently small in \eqref{est:404}
and combining \eqref{est:400}--\eqref{est:404} we arrive at
\begin{multline*}
\frac{d}{dt}\|e(t)\|^2_{L^2}
+
\|e(t)\|^2_{W^{1,2}}
\leq
\left( c_1 \|\bfu(t)\|^8_{\mathbf{L}^4} + c_2 \right) \|e(t)\|^2_{L^2}
\\
+
c_3
\left(
\| g(t)\|_{W^{-1,2}}
+
\|{\bfu}(t)\|^2_{\mathbf{W}^{1,12/5}}
\right)^2 .
\end{multline*}
Using the Gronwall's inequality yields
\begin{multline}\label{est:energyheat4}
\|e(t)\|^2_{L^2}
\leq
\exp\left(
\int_0^t
c_2
\|\bfu(s)\|^8_{\mathbf{L}^4}
+
c_3
\,{\rm d}s
\right)
\left[
\|e_0\|^2_{L^2}
\right.
\\
\left.
+
\int_0^t
c_1
\left(
\| g(s)\|_{W^{-1,2}}
+
\|{\bfu}(s)\|^2_{\mathbf{W}^{1,12/5}}
\right)^2
\,{\rm d}s
\right]
\end{multline}
for all $t \in I$. In view of \eqref{est:u_Y}
and the fact that $\tilde{\bfu} \in M$ we have
\begin{equation}\label{est_u_Y}
\|\bfu\|_{Y} \leq \frac{c}{2 C_S^2 T^{1/8}}.
\end{equation}
Therefore
\begin{eqnarray}\label{est:sup001}
\int_0^T
c_1
\|\bfu(s)\|^8_{\mathbf{L}^4}
+
c_2
\,{\rm d}s
& \leq &
c_1 \|\bfu\|^8_{\mathcal{C}(I,\mathbf{L}^4)} T + c_2 T
\nonumber
\\
& \leq &
c_1 \|\bfu\|^8_{Y} T + c_2 T
\nonumber
\\
& \leq &
c_1  + c_2 T.
\end{eqnarray}
Hence, combining \eqref{est:energyheat4}--\eqref{est:sup001} we can write
\begin{multline}\label{est:energyheat5}
\|e\|^2_{\mathcal{C}(I,L^2)}
\leq
\left[
\|e_0\|^2_{L^2}
+
c_1
\left(
\| g\|^2_{L^2(I,W^{-1,2})}
\right.
\right.
\\
\left.
\left.
+
\|{\bfu}\|^4_{L^4(I,\mathbf{W}^{1,12/5})}
\right)
\right]
\exp\left(
c_1  + c_2 T
\right).
\end{multline}
Now, combining \eqref{emb_160}, \eqref{emb:sup1} and \eqref{est_u_Y},
we deduce that $e$ satisfies the \textit{a priori} estimate of the general form
\begin{equation}\label{est:apriori_e}
\|e\|_{L^{2}({I};{L}^{2})}
\leq
T^{1/2}
\|e\|_{\mathcal{C}({I};{L}^{2})}
\leq
C(\Omega,T,C_S,\|e_0\|_{L^2},\| g\|_{L^2(I,W^{-1,2})}).
\end{equation}
We note that, by the \textit{a priori} estimate \eqref{est:apriori_e},
$e$ is bounded in $L^{2}({I};{L}^{2})$ independently of $\tilde{\bfu}$ and $\tilde{e}$.
We have previously shown that $\tilde{\bfu} \in M$ implies $\bfu \in M$.
Hence,  there exists a fixed ball $B \subset X \otimes L^2(I;L^2)$ defined by
$$
B
:=
\left\{
[\bfv,z ]\in X \otimes L^{2}({I};{L}^{2}), \;
\|\bfv\|_X
\leq
\frac{1}{2 C_S T^{1/8}},
\;
\|z\|_{L^{2}({I};{L}^{2})} \leq R
\right\}
$$
($R>0$ sufficiently large)
such that $\mathcal{T}(B)\subset B$, where the operator
$$
\mathcal{T}: X \otimes L^2(I;L^2) \rightarrow X \otimes L^2(I;L^2)
$$
is completely continuous.
Now, by Theorem \ref{th:Schauder}, there exists
the strong-weak solution $[\bfu,{e}]$ to problem
\eqref{eq3}--\eqref{init_temp}.
This completes the proof of the main result.
$\square$
\end{proof}

\begin{rem}
Let us explicitly note that, in view of  \eqref{main_cond}, we do not prescribe any ``smallness'' assumptions
on the norms $\|g\|_{L^2({I}; {W}^{-1,2})}$ and
$\|e_0\|_{L^{2}}$.
\end{rem}


\subsection*{Acknowledgment}
The first author of this work has been supported by the project GA\v{C}R 13-18652S.
The second author of this work has been supported by the \emph{Croatian Science Foundation} (scientific project 3955: \emph{Mathematical modeling and numerical simulations of processes in thin or porous domains}).

\end{document}